\documentclass[12pt]{amsart}
\usepackage{epsfig}
\usepackage{graphics}
\usepackage{color}
\usepackage{dcpic, pictexwd}
\usepackage{tikz}
\usepackage[utf8]{inputenc}
\usepackage{amsmath}
\usepackage{amsfonts}
\usepackage{amssymb}
\usepackage{calligra}
\usepackage{calrsfs}
\usepackage[mathscr]{euscript}
\usepackage{pinlabel}
\usepackage{xcolor}

\newtheorem{theorem}{Theorem}

\newtheorem{lemma}[theorem]{Lemma}

\newtheorem{corollary}[theorem]{Corollary}

\theoremstyle{definition}

\theoremstyle{remark}

\newcommand{\blankbox}[2]{%
  \parbox{\columnwidth}{\centering
    \setlength{\fboxsep}{0pt}%
    \fbox{\raisebox{0pt}[#2]{\hspace{#1}}}%
  }%
}

 \def\doublespacing{\parskip 5 pt plus 1 pt \baselineskip 25 pt
      \lineskip 13 pt \normallineskip 13 pt}

 \newcommand{\abs}[1]{\lvert#1\rvert}
 \def\R{{\mathbb{R}}}
 \def\Z{{\mathbb{Z}}}
 \def\Q{{\mathbb{Q}}}
 \def\N{{\mathbb{N}}}
 \def\S{{\Sigma}}
 \def\C{{\mathbb{C}}}
 \def\mcg{{\rm MCG}}
\def\mod{{\rm Mod}}
\def\emod{{\rm Mod}^*}
 \def\GL{{\rm GL}}
 \def\SL{{\rm SL}}
 \def\Sp{{\rm Sp}}
 \def\Im{{\rm Im}}
 \def\PSp{{\rm PSp}}
 \def\G{{\mathcal{G}}}
\def\S{{\mathcal{S}}}

\begin{document}

\newenvironment{prooff}{\medskip \par \noindent {\it Proof}\ }{\hfill
$\square$ \medskip \par}
    \def\sqr#1#2{{\vcenter{\hrule height.#2pt
        \hbox{\vrule width.#2pt height#1pt \kern#1pt
            \vrule width.#2pt}\hrule height.#2pt}}}
    \def\square{\mathchoice\sqr67\sqr67\sqr{2.1}6\sqr{1.5}6}
\def\pf#1{\medskip \par \noindent {\it #1.}\ }
\def\endpf{\hfill $\square$ \medskip \par}
\def\demo#1{\medskip \par \noindent {\it #1.}\ }
\def\enddemo{\medskip \par}
\def\qed{~\hfill$\square$}

 \title[Generating The Mapping Class Group By Three Involutions]
 {Generating The Mapping Class Group By Three Involutions}

\author[O\u{g}uz Y{\i}ld{\i}z]{O\u{g}uz Y{\i}ld{\i}z}
\date{\today}

 \address{Department of Mathematics, Middle East Technical University, 06800
  Ankara, Turkey
  }
 \email{oguzyildiz16@gmail.com}
\begin{abstract}
We prove that the mapping class group of a closed connected orientable surface of genus $g$ 
is generated by three involutions for $g\geq 6$. 
\end{abstract} 
 \maketitle

\section{Introduction}
The mapping class group $\mod(\Sigma_g)$ of a closed connected orientable 
surface $\Sigma_g$ of genus $g$ is the group of  
orientation--preserving diffeomorphisms of $\Sigma_g\to \Sigma_g$ up to isotopy. We are interested in
generating $\mod(\Sigma_g)$ by three involutions, elements of order two. 

We deal with the next attractive question suggested by Brendle-Farb and Kassabov. Korkmaz and Margalit gave place to the question in ~\cite{Korkmaz2012} and ~\cite{Margalit}: 
Is $\mod(\Sigma_g)$ generated by three involutions for $g \geq 3$?

Recently, Korkmaz~\cite{Korkmaz3inv} found a generating set consisting of three involutions for the mapping class group $\mod(\Sigma_g)$ for $g \geq 8$.
We proved on Theorem ~\ref{thm:1} that the mapping class group $\mod(\Sigma_g)$ is also generated by 
three involutions if $g \geq 6$.

\begin{theorem} \label{thm:1}
The mapping class group $\mod(\Sigma_g)$ is generated by 
three involutions if $g \geq 6$. 
\end{theorem}

We state the next corollary as in~\cite{Korkmaz3inv}. Since there is a surjective 
homomorphism from $\mod(\Sigma_g)$ onto the symplectic group 
$\Sp (2g,\Z)$, 
we have the following immediate result: 

\begin{corollary} 
The symplectic group $\Sp(2g,\Z)$ is generated 
by three involutions if $g\geq 6$.
\end{corollary}

Dehn~\cite{Dehn}
proved that $\mod(\Sigma_g)$ is generated by $2g(g-1)$ Dehn twists. Lickorish~\cite{Lickorish} 
found a generating set of $3g-1$ Dehn twists.
Humphries~\cite{Humphries} proved that $2g+1$ is the minimal number of Dehn twists that can generate $\mod(\Sigma_g)$ and presented a generating set of $2g+1$ Dehn twists.
Lu~\cite{Lu} obtained a generating set consisting of three elements, two of which are of finite order. 

Maclachlan~\cite{Maclachlan} was the first who showed that $\mod(\Sigma_g)$ can be generated by torsions. McCarthy and Papadopoulos~\cite{McCarthyPapadopoulos} proved that
$\mod(\Sigma_g)$ can be generated by only using involutions. Note that for homological reasons $\mod(\Sigma_g)$ is not generated by involutions if $g=1$ or $g=2$. Stukow~\cite{Stukow2}
showed that index five subgroup of $\mod(\Sigma_2)$ is generated by involutions. So, we consider the case of $g \geq 3$.
Luo~\cite{Luo} found an upper bound $12g+2$ for the number of involutions needed to generate $\mod(\Sigma_g)$.
Brendle and Farb~\cite{BrendleFarb} decreased this number to $6$. 
Kassabov~\cite{Kassabov} showed that four involutions are enough for $g\geq 7$. 
Recently, Korkmaz~\cite{Korkmaz3inv} found a generating set of three involutions if $g\geq 8$ and of four involutions if $g\geq 3$.



\section{An overview of mapping class groups}

Throughout the paper we only deal with closed connected orientable surfaces of genus g, $\Sigma_g$, 
as designed in Figure~\ref{fig1} so that the rotation $R$ by $2\pi$$/g$ degrees about $z$-axis is a well-defined
self-diffeomorphism of $\Sigma_g$. The mapping class group $\mod(\Sigma_g)$ of a closed connected orientable 
surface $\Sigma_g$ is the group of orientation--preserving diffeomorphisms of $\Sigma_g\to \Sigma_g$ up to isotopy.
Diffeomorphisms and curves shall be distinguished up to isotopy. For more detailed information on the mapping class groups,
one can be referred to~\cite{FarbMargalit}.
We only use three types of simple closed curves 
$a_i$'s, $b_i$'s and $c_i$'s as shown on Figure~\ref{fig1} where $1 \leq i \leq g$.
In order to align with the notation of Korkmaz in ~\cite{Korkmaz3inv}, 
we also denote simple closed curves by lowercase letters $a_i$, $b_i$, $c_i$
and corresponding right Dehn twists by uppercase letters $A_i$, $B_i$, $C_i$ or by commonly used notation $t_{a_i}$,$t_{b_i}$,$t_{c_i}$, respectively.
All indices shall be thought as modulo $g$.
Note that the composition of diffeomorphisms 
$f_1f_2$ means that $f_2$ is applied first and then $f_1$ is applied second.

For any simple closed curves $c_1$ and $c_2$ on $\Sigma_g$ and diffeomorphism $f:\Sigma_g \to \Sigma_g$, 
we use the following five basic facts along the paper for many times:

First, we have $ft_{c_1}f^{-1}=t_{f(c_1)}$.
Second, $c_1$ is isotopic to $c_2$ if and only if $t_{c_1}=t_{c_2}$ in $\mod(\Sigma_g)$.
Third, if $c_1$ and $c_2$ are disjoint, then $t_{c_1}(c_2)=c_2$.

\medskip 
\textbf{Commutativity:} If $c_1$ and $c_2$ are disjoint, then $t_{c_1}t_{c_2}=t_{c_2}t_{c_1}$ since
\begin{eqnarray*}
t_{c_1}t_{c_2}
&=& t_{c_2}t_{c_2}^{-1}t_{c_1}t_{c_2}\\
&=& t_{c_2}t_{t_{c_2}^{-1}(c_1)}\\
&=& t_{c_2}t_{c_1}.
\end{eqnarray*} 

\medskip
\textbf{Braid Relation:}   
If $c_1$ and $c_2$ are intersecting transversally at one point, then we have
$t_{c_1}t_{c_2}t_{c_1} = t_{c_2}t_{c_1}t_{c_2}$ since
\begin{eqnarray*}
(t_{c_1}t_{c_2})t_{c_1}(t_{c_1}t_{c_2})^{-1}
&=& t_{t_{c_1}t_{c_2}(c_1)}\\
&=& t_{c_2}.
\end{eqnarray*}

Dehn and Lickorish showed that the mapping class group $\mod(\Sigma_g)$ is generated by $3g-1$ Dehn twists about nonseparating
simple closed curves. After their work, Humphries~\cite{Humphries} proved the following theorem.
 
\begin{theorem} {\rm\bf(Dehn-Lickorish-Humphries)}\label{thm:Humphries}
The mapping class group $\mod(\Sigma_g)$ is generated by the set
$\{ A_1,A_2,B_1,B_2,\ldots ,B_g,C_1,C_2,\ldots ,C_{g-1}\}.$
\end{theorem}

Let $R$ be the rotation by $2\pi/g$ about the $x$--axis represented in Figure~\ref{fig1}. 
Then, \mbox{$R(a_k)=a_{k+1}$,} $R(b_k)=b_{k+1}$ and
$R(c_k)=c_{k+1}$. Korkmaz found a generating set of four elements, one of which is the rotation element $R$
and each of the other elements is a product of two opposite Dehn twists.
Korkmaz~\cite{Korkmaz3inv} proved the following useful theorem as a corollary to Theorem~\ref{thm:Humphries}.
\begin{theorem} \label{thm:thmKorkmaz} 
If $g\geq 3$, then the mapping class group $\mod(\Sigma_g)$ is generated by the four elements
\(
R, A_1A_2^{-1},B_1B_2^{-1}, C_1C_2^{-1}. \)
\end{theorem}
\begin{figure}
\begin{tikzpicture}[scale=0.5]
\begin{scope} [xshift=7cm]
 \draw[very thick, violet] (0,0) circle [radius=5.5cm];
 \draw[very thick, violet] (0,2.9) circle [radius=0.6cm]; 
 \draw[very thick, violet, rotate=60] (0,2.9) circle [radius=0.6cm]; 
 \draw[very thick, violet, rotate=-60] (0,2.9) circle [radius=0.6cm]; 
 \draw[very thick, violet, fill ] (0,-2.9) circle [radius=0.03cm]; 
 \draw[very thick, violet, fill, rotate=15] (0,-2.9) circle [radius=0.03cm]; 
\draw[very thick, violet, fill, rotate=-15] (0,-2.9) circle [radius=0.03cm]; 
\draw[thick, green,  rounded corners=10pt] (-0.05, 3.5) ..controls (-0.6,3.8) and (-0.6,5.2).. (-0.05,5.5) ;
\draw[thick, green, dashed, rounded corners=10pt] (0.05,3.5)..controls (0.6,3.8) and (0.6,5.2).. (0.05,5.5) ;
\node[scale=0.8] at (-1,4.7) {$a_1$};
\draw[thick, green, rotate=60, rounded corners=10pt] 
	(-0.05, 3.5) ..controls (-0.6,3.8) and (-0.6,5.2).. (-0.05,5.5) ;
\draw[thick, green, dashed, rotate=60, rounded corners=10pt] 
	(0.05,3.5)..controls (0.6,3.8) and (0.6,5.2).. (0.05,5.5) ;
\node[scale=0.8] at (-4.2,1.4) {$a_g$};
\draw[thick, green, rotate=-60, rounded corners=10pt] 
	(-0.05, 3.5) ..controls (-0.6,3.8) and (-0.6,5.2).. (-0.05,5.5) ;
\draw[thick, green, dashed, rotate=-60, rounded corners=10pt] 
	(0.05,3.5)..controls (0.6,3.8) and (0.6,5.2).. (0.05,5.5) ;
\node[scale=0.8] at (4.2,1.4) {$a_2$};
 \draw[thick, blue] (0,2.9) circle [radius=0.8cm]; 
 \draw[thick, rotate=60, blue] (0,2.9) circle [radius=0.8cm]; 
 \draw[thick, blue, rotate=-60] (0,2.9) circle [radius=0.8cm]; 
\node[scale=0.8] at (-2.6,2.7) {$b_g$};
\node[scale=0.8] at (2.7,2.6) {$b_2$};
\node[scale=0.8] at (1.0,3.5) {$b_1$};
\draw[thick, red, rounded corners=15pt] (-1.95, 1.65)-- (-1.756,2.633) -- (-0.6,2.68) ;
\draw[thick, red, dashed, rounded corners=15pt] (-1.83, 1.6)-- (-0.86,1.56) -- (-0.5,2.58) ;
\node[scale=0.8] at (-1.5,2.9) {$c_g$};
\draw[thick, red, rounded corners=15pt] (1.95, 1.65)-- (1.756,2.633) -- (0.6,2.68);
\draw[thick, red, dashed, rounded corners=15pt] (1.83, 1.6)-- (0.86,1.56) -- (0.5,2.58) ;
\node[scale=0.8] at (1.9,2.6) {$c_1$};
\draw[thick, red, rotate=60, rounded corners=10pt] (-1.9, 2.5)-- (-1.4,2.833) -- (-0.6,2.68) ;
\draw[thick, red, dashed, rotate=60, rounded corners=8pt] (-1.3, 1.9)-- (-0.8,1.96) -- (-0.5,2.58) ;
\node[scale=0.8] at (-3.7,0.0) {$c_{g-1}$};
\draw[thick, red, rotate=-60, rounded corners=10pt] (1.9, 2.5)-- (1.4,2.833) -- (0.6,2.68);
\draw[thick, red, dashed, rotate=-60, rounded corners=8pt] (1.3, 1.9)-- (0.8,1.96) -- (0.5,2.58) ;
\node[scale=0.8] at (3.5,0) {$c_2$};
\draw[->, rotate=-60] (-0.6,6.9)..controls (-0.8,6.2) and (0.8,6.2).. (0.6,6.9);
\draw[rotate=-60] (0,-7) -- (0,-5.6); \draw[rotate=-60] (0,2) -- (0,3.4); \draw[rotate=-60] (0,5.6) -- (0,6.25); \draw[rotate=-60] (0,6.5) -- (0,7.5);
\draw[dotted, rotate=-60] (0,-5.3) -- (0,1.85); \draw[dotted] (0,3.7) -- (0,5.3); 
\draw[->, rotate=-30] (-0.6,6.9)..controls (-0.8,6.2) and (0.8,6.2).. (0.6,6.9);
\draw[rotate=-30]  (0,-7) -- (0,-5.6);  \draw[rotate=-30] (0,5.6) -- (0,6.25); \draw[rotate=-30] (0,6.5) -- (0,7.5);
\draw[dotted, rotate=-30] (0,-5.3) -- (0,5.3); 
\node at (6.7,2.7) {$\rho_1$};
\node at (2.5,6.7) {$\rho_2$};
\draw[->, rounded corners=20pt, rotate=50] (0, 6.2)..controls (1.395, 6.039) and (2.716, 5.574).. (3.906,4.817);
\node[rotate=40] at (-3.3,5.8) {$R$};
\end{scope}
\node at (0.7+7, 6.7+0.5) {$y$};
\node at (6.7+7+0.5, 0.7) {$x$};
\draw[->, thick] (0+7,5.8)--(0+7, 8.3);
\draw[->, thick] (5.8+7,0)--(8.3+7,0);
\end{tikzpicture}
       \caption{The curves $a_i$, $b_i$, $c_i$, the involutions $\rho_1$, $\rho_2$ and the rotation $R$
          on the surface $\Sigma_g$.}
  \label{fig1}
\end{figure}
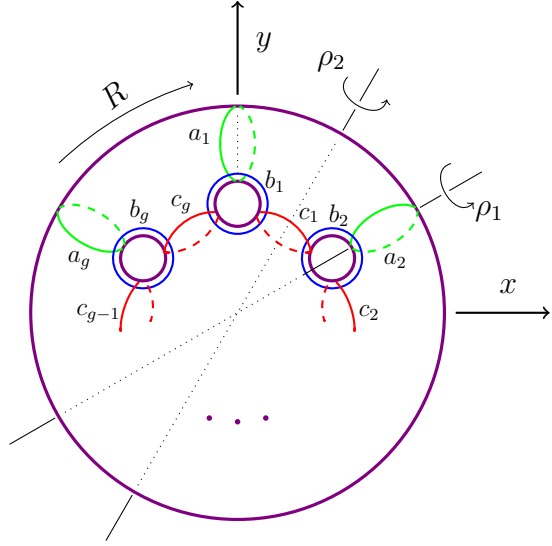

\section{Two new generating sets for $\mod(\Sigma_g)$.}
In this section, we introduce a new generating set for the mapping class group of $\Sigma_6$ and a new generating set for the mapping class group of $\Sigma_g$ for $g \geq 7$. 
We use these sets to generate the mapping class group by three involutions.

Define involutions $\rho_1$ and $\rho_2$ as rotations about the given axis on Figure~\ref{fig1} by $\pi$ degrees.
Recall that  $R$ denotes the $2\pi/g$--rotation of $\Sigma_g$ represented in
Figure~\ref{fig1}. It is a torsion element of order $g$ in the group $\mod(\Sigma_g)$.  
We follow the idea of Korkmaz in~\cite{Korkmaz3inv} to create our generating sets in the next lemmas as corollaries to Theorem~\ref{thm:thmKorkmaz}.
Our idea is to use $\rho_1$, $\rho_2$ as the first two elements and products of Dehn twists as the last element. Observe that $R = \rho_1 \rho_2$.

We use the next lemma (Lemma~\ref{lemma:1}) to create generating set of three involutions for $g=6$ and the second next lemma (Lemma~\ref{lemma:2}) to create generating set of three involutions for $g \geq 7$.
\begin{lemma} \label{lemma:1}
If $g=6$, then the mapping class group $\mod(\Sigma_g)$ is generated by 
three elements
$\rho_1$, $\rho_2$ and $A_1C_1B_3B_4^{-1}C_5^{-1}A_6^{-1}$. 
\end{lemma}

\begin{proof}
Let $F_1=A_1C_1B_3B_4^{-1}C_5^{-1}A_6^{-1}$. Let us denote by
$H$ the subgroup of $\mod(\Sigma_6)$ generated by the set
$\{ \rho_1, \rho_2, F_1\}$.

It suffices to show that $H$ contains $A_1A_2^{-1}$, $B_1B_2^{-1}$ and
$C_1C_2^{-1}$ by Theorem~\ref{thm:thmKorkmaz}.
Now, we define a list of elements to prove the desired result.

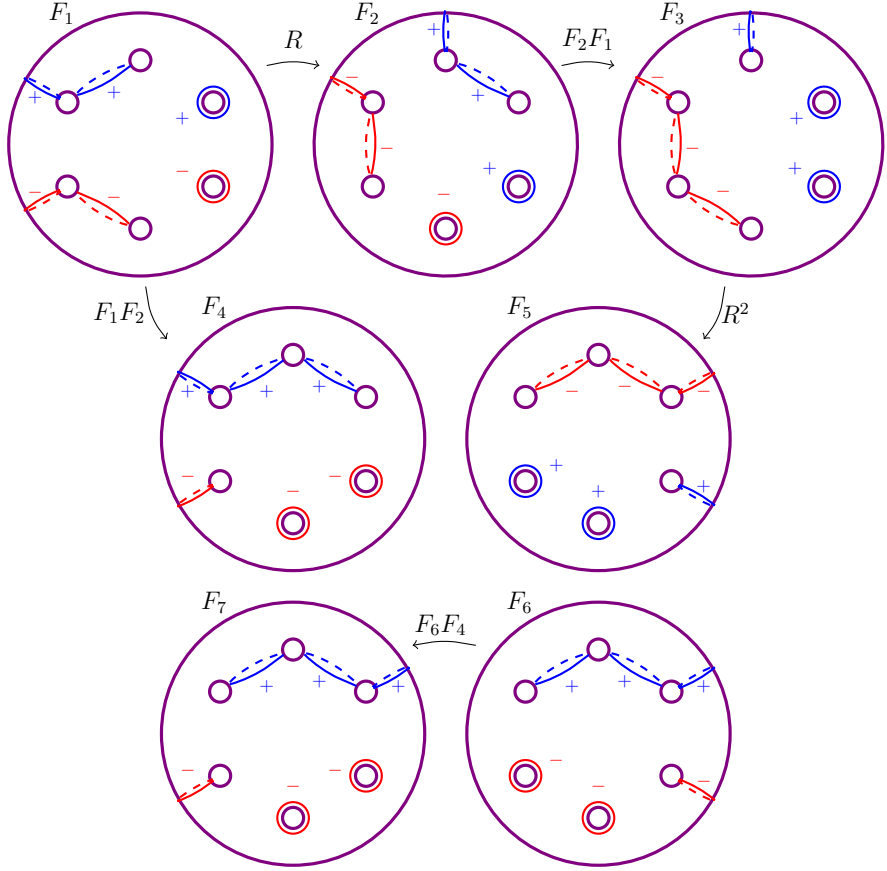
\begin{figure}
\begin{tikzpicture}[scale=0.7]
\begin{scope} [xshift=0cm, yshift=0cm]
 \draw[very thick, violet] (0,0) circle [radius=2.5cm];
 \draw[very thick, violet] (0,1.6) circle [radius=0.2cm]; 
 \draw[very thick, violet, rotate=60] (0,1.6) circle [radius=0.2cm]; 
 \draw[very thick, violet, rotate=120] (0,1.6) circle [radius=0.2cm];  
 \draw[very thick, violet, rotate=180] (0,1.6) circle [radius=0.2cm];  
 \draw[very thick, violet, rotate=-120] (0,1.6) circle [radius=0.2cm]; 
 \draw[very thick, violet, rotate=-60] (0,1.6) circle [radius=0.2cm]; 
 \draw[thick, blue, rotate=0, rounded corners=8pt] (-1.2,0.94)--(-0.8, 1.04)--(-0.18,1.5);  
 \draw[thick, blue, dashed,  rotate=0, rounded corners=8pt] (-1.2,0.98)--(-0.92, 1.28)--(-0.2,1.54);  
 \draw[thick, red, rotate=120, rounded corners=8pt] (-1.2,0.94)--(-0.8, 1.04)--(-0.18,1.5);  
 \draw[thick, red, dashed,  rotate=120, rounded corners=8pt] (-1.2,0.98)--(-0.92, 1.28)--(-0.2,1.54);  
 \draw[thick, blue, rotate=-60] (0,1.6) circle [radius=0.3cm]; 
\draw[thick, red, rotate=-120] (0,1.6) circle [radius=0.3cm]; 
 \draw[thick, blue, rotate=60, rounded corners=8pt] (-0.02,1.8)--(-0.1, 2.15)--(-0.02,2.5);  
 \draw[thick, blue, dashed, rotate=60, rounded corners=8pt]  (0.02,1.8)--(0.1, 2.15)--(0.02,2.5); 
 \draw[thick, dashed, red, rotate=120, rounded corners=8pt] (-0.02,1.8)--(-0.1, 2.15)--(-0.02,2.5);  
 \draw[thick, red, rotate=120, rounded corners=8pt]  (0.02,1.8)--(0.1, 2.15)--(0.02,2.5); 
  \node[scale=0.6, red] at (-2.0,-0.9) {$-$};
  \node[scale=0.6, blue] at (-0.5,1.0) {$+$};
  \node[scale=0.6, red] at (0.8,-0.5) {$-$};
  \node[scale=0.6, blue] at (-2.0,0.9) {$+$};
  \node[scale=0.6, blue] at (0.8,0.5) {$+$};
  \node[scale=0.6, red] at (-0.5,-1.0) {$-$};
\end{scope}

\begin{scope} [xshift=5.8cm, yshift=0cm, rotate=-60]
 \draw[very thick, violet] (0,0) circle [radius=2.5cm];
 \draw[very thick, violet] (0,1.6) circle [radius=0.2cm]; 
 \draw[very thick, violet, rotate=60] (0,1.6) circle [radius=0.2cm]; 
 \draw[very thick, violet, rotate=120] (0,1.6) circle [radius=0.2cm];  
 \draw[very thick, violet, rotate=180] (0,1.6) circle [radius=0.2cm];  
 \draw[very thick, violet, rotate=-120] (0,1.6) circle [radius=0.2cm]; 
 \draw[very thick, violet, rotate=-60] (0,1.6) circle [radius=0.2cm]; 
 \draw[thick, blue, rotate=0, rounded corners=8pt] (-1.2,0.94)--(-0.8, 1.04)--(-0.18,1.5);  
 \draw[thick, blue, dashed,  rotate=0, rounded corners=8pt] (-1.2,0.98)--(-0.92, 1.28)--(-0.2,1.54);  
 \draw[thick, red, rotate=120, rounded corners=8pt] (-1.2,0.94)--(-0.8, 1.04)--(-0.18,1.5);  
 \draw[thick, red, dashed,  rotate=120, rounded corners=8pt] (-1.2,0.98)--(-0.92, 1.28)--(-0.2,1.54);  
 \draw[thick, blue, rotate=-60] (0,1.6) circle [radius=0.3cm]; 
\draw[thick, red, rotate=-120] (0,1.6) circle [radius=0.3cm]; 
 \draw[thick, blue, rotate=60, rounded corners=8pt] (-0.02,1.8)--(-0.1, 2.15)--(-0.02,2.5);  
 \draw[thick, blue, dashed, rotate=60, rounded corners=8pt]  (0.02,1.8)--(0.1, 2.15)--(0.02,2.5); 
 \draw[thick, dashed, red, rotate=120, rounded corners=8pt] (-0.02,1.8)--(-0.1, 2.15)--(-0.02,2.5);  
 \draw[thick, red, rotate=120, rounded corners=8pt]  (0.02,1.8)--(0.1, 2.15)--(0.02,2.5); 
  \node[scale=0.6, red] at (-2.0,-0.9) {$-$};
  \node[scale=0.6, blue] at (-0.5,1.0) {$+$};
  \node[scale=0.6, red] at (0.8,-0.5) {$-$};
  \node[scale=0.6, blue] at (-2.0,0.9) {$+$};
  \node[scale=0.6, blue] at (0.8,0.5) {$+$};
  \node[scale=0.6, red] at (-0.5,-1.0) {$-$};
\end{scope}

\begin{scope} [xshift=11.6cm, yshift=0cm, rotate=-60]
 \draw[very thick, violet] (0,0) circle [radius=2.5cm];
 \draw[very thick, violet] (0,1.6) circle [radius=0.2cm]; 
 \draw[very thick, violet, rotate=60] (0,1.6) circle [radius=0.2cm]; 
 \draw[very thick, violet, rotate=120] (0,1.6) circle [radius=0.2cm];  
 \draw[very thick, violet, rotate=180] (0,1.6) circle [radius=0.2cm];  
 \draw[very thick, violet, rotate=-120] (0,1.6) circle [radius=0.2cm]; 
 \draw[very thick, violet, rotate=-60] (0,1.6) circle [radius=0.2cm]; 
 \draw[thick, red, rotate=180, rounded corners=8pt] (-1.2,0.94)--(-0.8, 1.04)--(-0.18,1.5);  
 \draw[thick, red, dashed,  rotate=180, rounded corners=8pt] (-1.2,0.98)--(-0.92, 1.28)--(-0.2,1.54);  
 \draw[thick, red, rotate=120, rounded corners=8pt] (-1.2,0.94)--(-0.8, 1.04)--(-0.18,1.5);  
 \draw[thick, red, dashed,  rotate=120, rounded corners=8pt] (-1.2,0.98)--(-0.92, 1.28)--(-0.2,1.54);  
 \draw[thick, blue, rotate=-60] (0,1.6) circle [radius=0.3cm]; 
\draw[thick, blue, rotate=0] (0,1.6) circle [radius=0.3cm]; 
 \draw[thick, blue, rotate=60, rounded corners=8pt] (-0.02,1.8)--(-0.1, 2.15)--(-0.02,2.5);  
 \draw[thick, blue, dashed, rotate=60, rounded corners=8pt]  (0.02,1.8)--(0.1, 2.15)--(0.02,2.5); 
 \draw[thick, dashed, red, rotate=120, rounded corners=8pt] (-0.02,1.8)--(-0.1, 2.15)--(-0.02,2.5);  
 \draw[thick, red, rotate=120, rounded corners=8pt]  (0.02,1.8)--(0.1, 2.15)--(0.02,2.5); 
  \node[scale=0.6, red] at (-2.0,-0.9) {$-$};
  \node[scale=0.6, blue] at (-0.0,1.0) {$+$};
  \node[scale=0.6, red] at (0.5,-0.9) {$-$};
  \node[scale=0.6, blue] at (-2.0,0.9) {$+$};
  \node[scale=0.6, blue] at (0.8,0.5) {$+$};
  \node[scale=0.6, red] at (-0.5,-1.0) {$-$};
\end{scope}

\begin{scope} [xshift=8.7cm, yshift=-5.6cm, rotate=-180]
 \draw[very thick, violet] (0,0) circle [radius=2.5cm];
 \draw[very thick, violet] (0,1.6) circle [radius=0.2cm]; 
 \draw[very thick, violet, rotate=60] (0,1.6) circle [radius=0.2cm]; 
 \draw[very thick, violet, rotate=120] (0,1.6) circle [radius=0.2cm];  
 \draw[very thick, violet, rotate=180] (0,1.6) circle [radius=0.2cm];  
 \draw[very thick, violet, rotate=-120] (0,1.6) circle [radius=0.2cm]; 
 \draw[very thick, violet, rotate=-60] (0,1.6) circle [radius=0.2cm]; 
 \draw[thick, red, rotate=180, rounded corners=8pt] (-1.2,0.94)--(-0.8, 1.04)--(-0.18,1.5);  
 \draw[thick, red, dashed,  rotate=180, rounded corners=8pt] (-1.2,0.98)--(-0.92, 1.28)--(-0.2,1.54);  
 \draw[thick, red, rotate=120, rounded corners=8pt] (-1.2,0.94)--(-0.8, 1.04)--(-0.18,1.5);  
 \draw[thick, red, dashed,  rotate=120, rounded corners=8pt] (-1.2,0.98)--(-0.92, 1.28)--(-0.2,1.54);  
 \draw[thick, blue, rotate=-60] (0,1.6) circle [radius=0.3cm]; 
\draw[thick, blue, rotate=0] (0,1.6) circle [radius=0.3cm]; 
 \draw[thick, blue, rotate=60, rounded corners=8pt] (-0.02,1.8)--(-0.1, 2.15)--(-0.02,2.5);  
 \draw[thick, blue, dashed, rotate=60, rounded corners=8pt]  (0.02,1.8)--(0.1, 2.15)--(0.02,2.5); 
 \draw[thick, dashed, red, rotate=120, rounded corners=8pt] (-0.02,1.8)--(-0.1, 2.15)--(-0.02,2.5);  
 \draw[thick, red, rotate=120, rounded corners=8pt]  (0.02,1.8)--(0.1, 2.15)--(0.02,2.5); 
  \node[scale=0.6, red] at (-2.0,-0.9) {$-$};
  \node[scale=0.6, blue] at (-0.0,1.0) {$+$};
  \node[scale=0.6, red] at (0.5,-0.9) {$-$};
  \node[scale=0.6, blue] at (-2.0,0.9) {$+$};
  \node[scale=0.6, blue] at (0.8,0.5) {$+$};
  \node[scale=0.6, red] at (-0.5,-1.0) {$-$};  
\end{scope}

\begin{scope} [xshift=2.9cm, yshift=-5.6cm, rotate=180]
 \draw[very thick, violet] (0,0) circle [radius=2.5cm];
 \draw[very thick, violet] (0,1.6) circle [radius=0.2cm]; 
 \draw[very thick, violet, rotate=60] (0,1.6) circle [radius=0.2cm]; 
 \draw[very thick, violet, rotate=120] (0,1.6) circle [radius=0.2cm];  
 \draw[very thick, violet, rotate=180] (0,1.6) circle [radius=0.2cm];  
 \draw[very thick, violet, rotate=-120] (0,1.6) circle [radius=0.2cm]; 
 \draw[very thick, violet, rotate=-60] (0,1.6) circle [radius=0.2cm]; 
 \draw[thick, blue, rotate=180, rounded corners=8pt] (-1.2,0.94)--(-0.8, 1.04)--(-0.18,1.5);  
 \draw[thick, blue, dashed,  rotate=180, rounded corners=8pt] (-1.2,0.98)--(-0.92, 1.28)--(-0.2,1.54);  
 \draw[thick, blue, rotate=120, rounded corners=8pt] (-1.2,0.94)--(-0.8, 1.04)--(-0.18,1.5);  
 \draw[thick, blue, dashed,  rotate=120, rounded corners=8pt] (-1.2,0.98)--(-0.92, 1.28)--(-0.2,1.54);  
 \draw[thick, red, rotate=60] (0,1.6) circle [radius=0.3cm]; 
\draw[thick, red, rotate=0] (0,1.6) circle [radius=0.3cm]; 
 \draw[thick, red, rotate=-60, rounded corners=8pt] (-0.02,1.8)--(-0.1, 2.15)--(-0.02,2.5);  
 \draw[thick, red, dashed, rotate=-60, rounded corners=8pt]  (0.02,1.8)--(0.1, 2.15)--(0.02,2.5); 
 \draw[thick, dashed, blue, rotate=-120, rounded corners=8pt] (-0.02,1.8)--(-0.1, 2.15)--(-0.02,2.5);  
 \draw[thick, blue, rotate=-120, rounded corners=8pt]  (0.02,1.8)--(0.1, 2.15)--(0.02,2.5); 
  \node[scale=0.6, blue] at (2.0,-0.9) {$+$};
  \node[scale=0.6, red] at (-0.0,1.0) {$-$};
  \node[scale=0.6, blue] at (0.5,-0.9) {$+$};
  \node[scale=0.6, red] at (-0.8,0.7) {$-$};
  \node[scale=0.6, red] at (2.0,0.7) {$-$};
  \node[scale=0.6, blue] at (-0.5,-1.0) {$+$};
\end{scope}

\begin{scope} [xshift=8.7cm, yshift=-11.2cm, rotate=-180]
 \draw[very thick, violet] (0,0) circle [radius=2.5cm];
 \draw[very thick, violet] (0,1.6) circle [radius=0.2cm]; 
 \draw[very thick, violet, rotate=60] (0,1.6) circle [radius=0.2cm]; 
 \draw[very thick, violet, rotate=120] (0,1.6) circle [radius=0.2cm];  
 \draw[very thick, violet, rotate=180] (0,1.6) circle [radius=0.2cm];  
 \draw[very thick, violet, rotate=-120] (0,1.6) circle [radius=0.2cm]; 
 \draw[very thick, violet, rotate=-60] (0,1.6) circle [radius=0.2cm]; 
 \draw[thick, blue, rotate=180, rounded corners=8pt] (-1.2,0.94)--(-0.8, 1.04)--(-0.18,1.5);  
 \draw[thick, blue, dashed,  rotate=180, rounded corners=8pt] (-1.2,0.98)--(-0.92, 1.28)--(-0.2,1.54);  
 \draw[thick, blue, rotate=120, rounded corners=8pt] (-1.2,0.94)--(-0.8, 1.04)--(-0.18,1.5);  
 \draw[thick, blue, dashed,  rotate=120, rounded corners=8pt] (-1.2,0.98)--(-0.92, 1.28)--(-0.2,1.54);  
 \draw[thick, red, rotate=-60] (0,1.6) circle [radius=0.3cm]; 
\draw[thick, red, rotate=0] (0,1.6) circle [radius=0.3cm]; 
 \draw[thick, red, rotate=60, rounded corners=8pt] (-0.02,1.8)--(-0.1, 2.15)--(-0.02,2.5);  
 \draw[thick, red, dashed, rotate=60, rounded corners=8pt]  (0.02,1.8)--(0.1, 2.15)--(0.02,2.5); 
 \draw[thick, dashed, blue, rotate=120, rounded corners=8pt] (-0.02,1.8)--(-0.1, 2.15)--(-0.02,2.5);  
 \draw[thick, blue, rotate=120, rounded corners=8pt]  (0.02,1.8)--(0.1, 2.15)--(0.02,2.5); 
  \node[scale=0.6, blue] at (-2.0,-0.9) {$+$};
  \node[scale=0.6, red] at (-0.0,1.0) {$-$};
  \node[scale=0.6, blue] at (0.5,-0.9) {$+$};
  \node[scale=0.6, red] at (-2.0,0.9) {$-$};
  \node[scale=0.6, red] at (0.8,0.5) {$-$};
  \node[scale=0.6, blue] at (-0.5,-1.0) {$+$}; 
\end{scope}

\begin{scope} [xshift=2.9cm, yshift=-11.2cm, rotate=180]
\draw[very thick, violet] (0,0) circle [radius=2.5cm];
 \draw[very thick, violet] (0,1.6) circle [radius=0.2cm]; 
 \draw[very thick, violet, rotate=60] (0,1.6) circle [radius=0.2cm]; 
 \draw[very thick, violet, rotate=120] (0,1.6) circle [radius=0.2cm];  
 \draw[very thick, violet, rotate=180] (0,1.6) circle [radius=0.2cm];  
 \draw[very thick, violet, rotate=-120] (0,1.6) circle [radius=0.2cm]; 
 \draw[very thick, violet, rotate=-60] (0,1.6) circle [radius=0.2cm]; 
 \draw[thick, blue, rotate=180, rounded corners=8pt] (-1.2,0.94)--(-0.8, 1.04)--(-0.18,1.5);  
 \draw[thick, blue, dashed,  rotate=180, rounded corners=8pt] (-1.2,0.98)--(-0.92, 1.28)--(-0.2,1.54);  
 \draw[thick, blue, rotate=120, rounded corners=8pt] (-1.2,0.94)--(-0.8, 1.04)--(-0.18,1.5);  
 \draw[thick, blue, dashed,  rotate=120, rounded corners=8pt] (-1.2,0.98)--(-0.92, 1.28)--(-0.2,1.54);  
 \draw[thick, red, rotate=60] (0,1.6) circle [radius=0.3cm]; 
\draw[thick, red, rotate=0] (0,1.6) circle [radius=0.3cm]; 
 \draw[thick, red, rotate=-60, rounded corners=8pt] (-0.02,1.8)--(-0.1, 2.15)--(-0.02,2.5);  
 \draw[thick, red, dashed, rotate=-60, rounded corners=8pt]  (0.02,1.8)--(0.1, 2.15)--(0.02,2.5); 
 \draw[thick, dashed, blue, rotate=120, rounded corners=8pt] (-0.02,1.8)--(-0.1, 2.15)--(-0.02,2.5);  
 \draw[thick, blue, rotate=120, rounded corners=8pt]  (0.02,1.8)--(0.1, 2.15)--(0.02,2.5); 
  \node[scale=0.6, blue] at (-2.0,-0.9) {$+$};
  \node[scale=0.6, red] at (-0.0,1.0) {$-$};
  \node[scale=0.6, blue] at (0.5,-0.9) {$+$};
  \node[scale=0.6, red] at (-0.8,0.7) {$-$};
  \node[scale=0.6, red] at (2.0,0.7) {$-$};
  \node[scale=0.6, blue] at (-0.5,-1.0) {$+$}; 
 \end{scope}

	 \node[scale=0.8] at (-1.5,2.5) {$F_1$};
 	\node[scale=0.8] at (4.3,2.5) {$F_2$};
 	\node[scale=0.8] at (10.1,2.5) {$F_3$};
 	\node[scale=0.8] at (7.2,-3.1) {$F_5$};
 	\node[scale=0.8] at (1.4,-3.1) {$F_4$};
 	\node[scale=0.8] at (1.4,-8.7) {$F_7$};
 	\node[scale=0.8] at (7.2,-8.7) {$F_6$};
 	\draw[ ->, rounded corners=5pt] (2.4,1.5)--(2.9,1.6)-- (3.4,1.5);
 	\node[scale=0.8] at (2.9,2) {$R$};
 	\draw[ ->, rounded corners=5pt] (2.4+5.6,1.5)--(2.9+5.6,1.6)-- (3.4+5.6,1.5);
 	\node[scale=0.8] at (2.9+5.6,2) {$F_2F_1$};
 	\draw[ ->, rounded corners=5pt] (0.1,-2.7)--(0.2, -3.3)--(0.5,-3.7);
 	\node[scale=0.8] at (-0.4,-3.2) {$F_1F_2$};
 	\draw[ ->, rounded corners=5pt] (11.1,-2.7)--(11, -3.3)--(10.7,-3.7);
 	\node[scale=0.8] at (11.3,-3.2) {$R^2$};
 	\draw[ ->, rounded corners=5pt] (5.7+0.67,-3.92-5.6)--(5.7+0.07,-3.82-5.6)--(5.7-0.53,-3.92-5.6) ;
 	\node[scale=0.8] at (5.7,-3.53-5.6) {$F_6F_4$};
\end{tikzpicture}
 	\caption{Proof of Lemma~\ref{lemma:1}.}
\end{figure}

Let \begin{eqnarray*}
F_2
&=& RF_1R^{-1} \\
&=& R(A_1C_1B_3B_4^{-1}C_5^{-1}A_6^{-1})R^{-1}\\
&=& RA_1R^{-1}RC_1R^{-1}RB_3R^{-1}RB_4^{-1}R^{-1}RC_5^{-1}R^{-1}RA_6^{-1}R^{-1}\\
&=& Rt_{a_1}R^{-1}Rt_{c_1}R^{-1}Rt_{b_3}R^{-1}Rt_{b_4}^{-1}R^{-1}Rt_{c_5}^{-1}R^{-1}Rt_{a_6}^{-1}R^{-1}\\
&=& t_{R(a_1)}t_{R(c_1)}t_{R(b_3)}t_{R(b_4)}^{-1}t_{R(c_5)}^{-1}t_{R(a_6)}^{-1}\\
&=& t_{a_2}t_{c_2}t_{b_4}t_{b_5}^{-1}t_{c_6}^{-1}t_{a_1}^{-1}\\
&=& A_2C_2B_4B_5^{-1}C_6^{-1}A_1^{-1}.
\end{eqnarray*}

Then, $F_2F_1(a_2,c_2,b_4,b_5,c_6,a_1)=(a_2,b_3,b_4,c_5,c_6,a_1)$ so that\\
$F_3=A_2B_3B_4C_5^{-1}C_6^{-1}A_1^{-1}\in H$. Note that $F_2F_1(c_2) = b_3$ 
since \begin{eqnarray*}
t_{F_2F_1(c_2)}
&=& (F_2F_1)t_{c_2}(F_2F_1)^{-1}\\
&=& F_2F_1C_2F_1^{-1}F_2^{-1}\\
&=& C_2B_3C_2B_3^{-1}C_2^{-1} \\
&=& (t_{c_2}t_{b_3})t_{c_2}(t_{c_2}t_{b_3})^{-1}\\
&=& t_{t_{c_2}t_{b_3}(c_2)}\\
&=& t_{b_3}.
\end{eqnarray*}

$F_1F_2(a_1,c_1,b_3,b_4,c_5,a_6)=(a_1,c_1,c_2,b_4,b_5,a_6)$ so that\\
$F_4=A_1C_1C_2B_4^{-1}B_5^{-1}A_6^{-1}\in H$.

Let
\[
F_5 = R^2F_3R^{-2} = A_4B_5B_6C_1^{-1}C_2^{-1}A_3^{-1}
\]
and
\[
F_6 = F_5^{-1} = A_3C_2C_1B_6^{-1}B_5^{-1}A_4^{-1}.
\]

Then, $F_6F_4(a_3,c_2,c_1,b_6,b_5,a_4)=(a_3,c_2,c_1,a_6,b_5,b_4)$ so that\\
$F_7=A_3C_2C_1A_6^{-1}B_5^{-1}B_4^{-1}\in H$.

Let 
\[
F_8 = F_4F_7^{-1} = A_1A_3^{-1}.
\]

$F_8^{-1}F_3(a_1,a_3)=(a_1,b_3)$ so that $A_1B_3^{-1}$$\in$$H$ and then by conjugating $A_1B_3^{-1}$ with $R$ iteratively, we get $A_iB_{i+2}^{-1}$$\in$$H$ $\forall i$.

Let 
\[
F_9 = B_3B_5^{-1} = (B_3A_1^{-1})(A_1A_3^{-1})(A_3B_5^{-1}).
\]

$F_9F_3(b_3,b_5)=(b_3,c_5)$ so that $B_3C_5^{-1}$$\in$$H$ and then \mbox{$B_iC_{i+2}^{-1}$$\in$$H$ $\forall i$}.\\
$\rho_1(B_2C_4^{-1})\rho_1=B_2C_5^{-1}$$\in$$H$ and then \mbox{$B_iC_{i+3}^{-1}$$\in$$H$ $\forall i$}.\\
$B_1B_2^{-1}=(B_1C_4^{-1})(C_4B_2^{-1})$$\in$$H$ and then $B_3B_4^{-1}$$\in$$H$.\\
$A_1A_2^{-1}=(A_1B_3^{-1})(B_3B_4^{-1})(B_4A_2^{-1})$$\in$$H$.\\
$C_1C_2^{-1}=(C_1B_5^{-1})(B_5C_2^{-1})$$\in$$H$.

It follows from  Theorem~\ref{thm:thmKorkmaz} that
$H=\mod(\Sigma_6)$, completing the proof of the lemma.
 \end{proof}

 
\begin{lemma} \label{lemma:2}
If $g \geq 7$, then the mapping class group  $\mod(\Sigma_g)$ is generated by 
three elements
$\rho_1$, $\rho_2$ and $A_1C_1B_3B_5^{-1}C_6^{-1}A_7^{-1}$. 
\end{lemma}

\begin{proof}
We follow a similar pattern to Lemma ~\ref{lemma:1}.

Let $F_1=A_1C_1B_3B_5^{-1}C_6^{-1}A_7^{-1}$ and
$H$ be the subgroup of $\mod(\Sigma_g)$ generated by the set
$\{ \rho_1, \rho_2, F_1\}$. It suffices to show that $H$ contains $A_1A_2^{-1}$, $B_1B_2^{-1}$ and
$C_1C_2^{-1}$. 
Assume $g=7$. \textcolor{red}{(Respectively, $g \geq 8$.)}

Let
\[
F_2 = RF_1R^{-1} = A_2C_2B_4B_6^{-1}C_7^{-1}A_1^{-1}. 
\] \textcolor{red}{(Resp., $F_2=A_2C_2B_4B_6^{-1}C_7^{-1}A_8^{-1}$).}
 
Then, $F_2F_1(a_2,c_2,b_4,b_6,c_7,a_1)=(a_2,b_3,b_4,c_6,c_7,a_1)$ so that\\
$F_3=A_2B_3B_4C_6^{-1}C_7^{-1}A_1^{-1}\in H$. \textcolor{red}{(Resp., $F_3=A_2B_3B_4C_6^{-1}C_7^{-1}A_8^{-1}$).}

Let \[
F_4 = R^{-1}F_3R = A_1B_2B_3C_5^{-1}C_6^{-1}A_7^{-1}.
\]

Then, $F_4F_3(a_1,b_2,b_3,c_5,c_6,a_7)=(a_1,a_2,b_3,c_5,c_6,a_7)$ so that\\
$F_5=A_1A_2B_3C_5^{-1}C_6^{-1}A_7^{-1}\in H$.

$F_4F_5^{-1} = B_2A_2^{-1}$ so that $B_2A_2^{-1}$$\in$$H$ and then by conjugating $B_2A_2^{-1}$ with $R$ iteratively, we get $B_iA_i^{-1}$$\in$$H$ $\forall i$.

Let \[
F_6
= (A_3B_3^{-1})F_1 = A_1C_1A_3B_5^{-1}C_6^{-1}A_7^{-1}
\] and
\[
F_7 = RF_6R^{-1} = A_2C_2A_4B_6^{-1}C_7^{-1}A_1^{-1}. 
\] \textcolor{red}{(Resp., $F_7=A_2C_2A_4B_6^{-1}C_7^{-1}A_8^{-1}$).}

Then, $F_7F_6(a_2,c_2,a_4,b_6,c_7,a_1)=(a_2,c_2,a_4,c_6,c_7,a_1)$ so that\\
$F_8=A_2C_2A_4C_6^{-1}C_7^{-1}A_1^{-1}\in H$. \textcolor{red}{(Resp., $F_8=A_2C_2A_4C_6^{-1}C_7^{-1}A_8^{-1}$).} 

$F_7^{-1}F_8=B_6C_6^{-1}$ so that $B_6C_6^{-1}$$\in$$H$ and then \mbox{$B_iC_i^{-1}$$\in$$H$ $\forall i$}.\\
$\rho_2(B_1C_1^{-1})\rho_2=B_2C_1^{-1}$$\in$$H$ and then \mbox{$B_{i+1}C_i^{-1}$$\in$$H$ $\forall i$}.\\
$C_1C_2^{-1}=(C_1B_2^{-1})(B_2C_2^{-1})$$\in$$H$.\\
$B_1B_2^{-1}=(B_1C_1^{-1})(C_1C_2^{-1})(C_2B_2^{-1})$$\in$$H$ and then $B_3B_4^{-1}$$\in$$H$.\\
$A_1A_2^{-1}=(A_1B_1^{-1})(B_1B_2^{-1})(B_2A_2^{-1})$$\in$$H$.


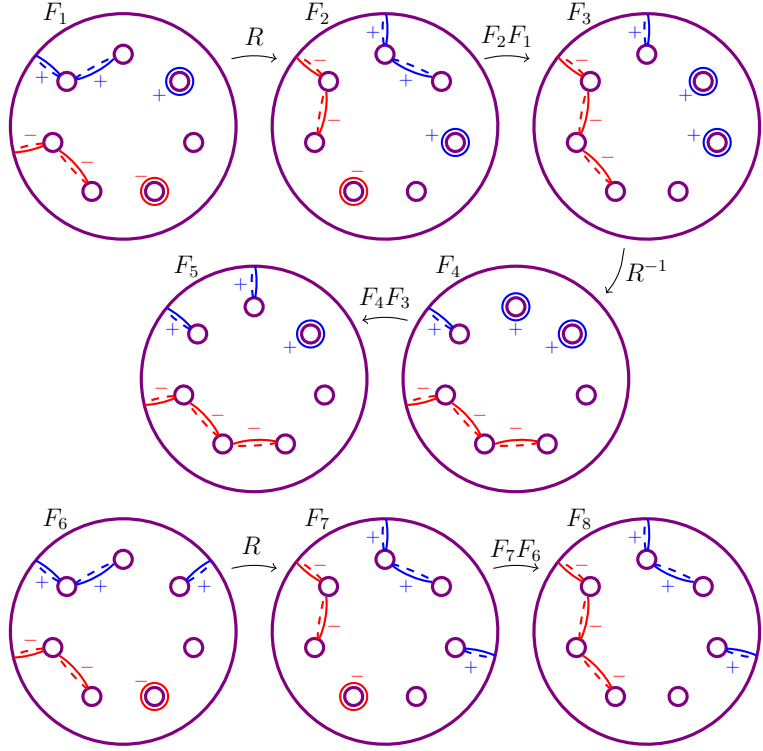
\begin{figure}
\begin{tikzpicture}[scale=0.6]
\begin{scope} [xshift=0cm, yshift=0cm]
 \draw[very thick, violet] (0,0) circle [radius=2.5cm];
 \draw[very thick, violet] (0,1.6) circle [radius=0.2cm]; 
 \draw[very thick, violet, rotate=360/7] (0,1.6) circle [radius=0.2cm]; 
 \draw[very thick, violet, rotate=360/7*2] (0,1.6) circle [radius=0.2cm];  
 \draw[very thick, violet, rotate=360/7*3] (0,1.6) circle [radius=0.2cm];  
 \draw[very thick, violet, rotate=360/7*4] (0,1.6) circle [radius=0.2cm]; 
 \draw[very thick, violet, rotate=360/7*5] (0,1.6) circle [radius=0.2cm]; 
 \draw[very thick, violet, rotate=360/7*6] (0,1.6) circle [radius=0.2cm]; 
 \draw[thick, blue, rotate=0, rounded corners=8pt](-1.04,1.04)--(-0.57, 1.14)--(-0.21,1.5);   
 \draw[thick, blue, dashed,  rotate=0, rounded corners=8pt] (-1.04,1.08)--(-0.69, 1.35)--(-0.2,1.54);  
 \draw[thick, red, rotate=360/7*2, rounded corners=8pt](-1.04,1.04)--(-0.57, 1.14)--(-0.21,1.5);  
 \draw[thick, red, dashed,  rotate=360/7*2, rounded corners=8pt](-1.04,1.08)--(-0.69, 1.35)--(-0.2,1.54);    
 \draw[thick, blue, rotate=360/7*6] (0,1.6) circle [radius=0.3cm]; 
\draw[thick, red, rotate=360/7*4] (0,1.6) circle [radius=0.3cm]; 
 \draw[thick, blue, rotate=360/7, rounded corners=6pt] (0.02,1.8)--(0.1, 2.15)--(0.02,2.5);  
 \draw[thick, blue, dashed, rotate=360/7, rounded corners=6pt]   (-0.02,1.8)--(-0.1, 2.15)--(-0.02,2.5);  
 \draw[thick, dashed, red, rotate=360/7*2, rounded corners=6pt]  (0.02,1.8)--(0.1, 2.15)--(0.02,2.5);  
 \draw[thick, red, rotate=360/7*2, rounded corners=6pt]   (-0.02,1.8)--(-0.1, 2.15)--(-0.02,2.5);  
  \node[scale=0.6, red] at (-2.05,-0.2) {$-$};
  \node[scale=0.6, blue] at (0.8,0.7) {$+$};
  \node[scale=0.6, red] at (0.4,-1.1) {$-$};
  \node[scale=0.6, blue] at (-1.8,1.1) {$+$};
  \node[scale=0.6, blue] at (-0.5,1.0) {$+$};
  \node[scale=0.6, red] at (-0.8,-0.8) {$-$};
\end{scope}

\begin{scope} [xshift=5.8cm, yshift=0cm, rotate=-360/7]
  \draw[very thick, violet] (0,0) circle [radius=2.5cm];
 \draw[very thick, violet] (0,1.6) circle [radius=0.2cm]; 
 \draw[very thick, violet, rotate=360/7] (0,1.6) circle [radius=0.2cm]; 
 \draw[very thick, violet, rotate=360/7*2] (0,1.6) circle [radius=0.2cm];  
 \draw[very thick, violet, rotate=360/7*3] (0,1.6) circle [radius=0.2cm];  
 \draw[very thick, violet, rotate=360/7*4] (0,1.6) circle [radius=0.2cm]; 
 \draw[very thick, violet, rotate=360/7*5] (0,1.6) circle [radius=0.2cm]; 
 \draw[very thick, violet, rotate=360/7*6] (0,1.6) circle [radius=0.2cm]; 
 \draw[thick, blue, rotate=0, rounded corners=8pt](-1.04,1.04)--(-0.57, 1.14)--(-0.21,1.5);   
 \draw[thick, blue, dashed,  rotate=0, rounded corners=8pt] (-1.04,1.08)--(-0.69, 1.35)--(-0.2,1.54);  
 \draw[thick, red, rotate=360/7*2, rounded corners=8pt](-1.04,1.04)--(-0.57, 1.14)--(-0.21,1.5);  
 \draw[thick, red, dashed,  rotate=360/7*2, rounded corners=8pt](-1.04,1.08)--(-0.69, 1.35)--(-0.2,1.54);    
 \draw[thick, blue, rotate=360/7*6] (0,1.6) circle [radius=0.3cm]; 
\draw[thick, red, rotate=360/7*4] (0,1.6) circle [radius=0.3cm]; 
 \draw[thick, blue, rotate=360/7, rounded corners=6pt] (0.02,1.8)--(0.1, 2.15)--(0.02,2.5);  
 \draw[thick, blue, dashed, rotate=360/7, rounded corners=6pt]   (-0.02,1.8)--(-0.1, 2.15)--(-0.02,2.5);  
 \draw[thick, dashed, red, rotate=360/7*2, rounded corners=6pt]  (0.02,1.8)--(0.1, 2.15)--(0.02,2.5);  
 \draw[thick, red, rotate=360/7*2, rounded corners=6pt]   (-0.02,1.8)--(-0.1, 2.15)--(-0.02,2.5);  
  \node[scale=0.6, red] at (-2.05,-0.2) {$-$};
  \node[scale=0.6, blue] at (0.8,0.7) {$+$};
  \node[scale=0.6, red] at (0.4,-1.1) {$-$};
  \node[scale=0.6, blue] at (-1.8,1.1) {$+$};
  \node[scale=0.6, blue] at (-0.5,1.0) {$+$};
  \node[scale=0.6, red] at (-0.8,-0.8) {$-$};
\end{scope}

\begin{scope} [xshift=11.6cm, yshift=0cm, rotate=-360/7]
  \draw[very thick, violet] (0,0) circle [radius=2.5cm];
 \draw[very thick, violet] (0,1.6) circle [radius=0.2cm]; 
 \draw[very thick, violet, rotate=360/7] (0,1.6) circle [radius=0.2cm]; 
 \draw[very thick, violet, rotate=360/7*2] (0,1.6) circle [radius=0.2cm];  
 \draw[very thick, violet, rotate=360/7*3] (0,1.6) circle [radius=0.2cm];  
 \draw[very thick, violet, rotate=360/7*4] (0,1.6) circle [radius=0.2cm]; 
 \draw[very thick, violet, rotate=360/7*5] (0,1.6) circle [radius=0.2cm]; 
 \draw[very thick, violet, rotate=360/7*6] (0,1.6) circle [radius=0.2cm]; 
 \draw[thick, red, rotate=360/7*3, rounded corners=8pt](-1.04,1.04)--(-0.57, 1.14)--(-0.21,1.5);   
 \draw[thick, red, dashed,  rotate=360/7*3, rounded corners=8pt] (-1.04,1.08)--(-0.69, 1.35)--(-0.2,1.54);  
 \draw[thick, red, rotate=360/7*2, rounded corners=8pt](-1.04,1.04)--(-0.57, 1.14)--(-0.21,1.5);  
 \draw[thick, red, dashed,  rotate=360/7*2, rounded corners=8pt](-1.04,1.08)--(-0.69, 1.35)--(-0.2,1.54);    
 \draw[thick, blue, rotate=360/7*6] (0,1.6) circle [radius=0.3cm]; 
\draw[thick, blue, rotate=0] (0,1.6) circle [radius=0.3cm]; 
 \draw[thick, blue, rotate=360/7, rounded corners=6pt] (0.02,1.8)--(0.1, 2.15)--(0.02,2.5);  
 \draw[thick, blue, dashed, rotate=360/7, rounded corners=6pt]   (-0.02,1.8)--(-0.1, 2.15)--(-0.02,2.5);  
 \draw[thick, dashed, red, rotate=360/7*2, rounded corners=6pt]  (0.02,1.8)--(0.1, 2.15)--(0.02,2.5);  
 \draw[thick, red, rotate=360/7*2, rounded corners=6pt]   (-0.02,1.8)--(-0.1, 2.15)--(-0.02,2.5);  
  \node[scale=0.6, red] at (-2.05,-0.2) {$-$};
  \node[scale=0.6, blue] at (0.8,0.7) {$+$};
  \node[scale=0.6, red] at (0.0,-1.1) {$-$};
  \node[scale=0.6, blue] at (-1.8,1.1) {$+$};
  \node[scale=0.6, blue] at (-0.0,1.1) {$+$};
  \node[scale=0.6, red] at (-0.8,-0.8) {$-$};
\end{scope}

\begin{scope} [xshift=8.7cm, yshift=-5.6cm, rotate=0]
  \draw[very thick, violet] (0,0) circle [radius=2.5cm];
 \draw[very thick, violet] (0,1.6) circle [radius=0.2cm]; 
 \draw[very thick, violet, rotate=360/7] (0,1.6) circle [radius=0.2cm]; 
 \draw[very thick, violet, rotate=360/7*2] (0,1.6) circle [radius=0.2cm];  
 \draw[very thick, violet, rotate=360/7*3] (0,1.6) circle [radius=0.2cm];  
 \draw[very thick, violet, rotate=360/7*4] (0,1.6) circle [radius=0.2cm]; 
 \draw[very thick, violet, rotate=360/7*5] (0,1.6) circle [radius=0.2cm]; 
 \draw[very thick, violet, rotate=360/7*6] (0,1.6) circle [radius=0.2cm]; 
 \draw[thick, red, rotate=360/7*3, rounded corners=8pt](-1.04,1.04)--(-0.57, 1.14)--(-0.21,1.5);   
 \draw[thick, red, dashed,  rotate=360/7*3, rounded corners=8pt] (-1.04,1.08)--(-0.69, 1.35)--(-0.2,1.54);  
 \draw[thick, red, rotate=360/7*2, rounded corners=8pt](-1.04,1.04)--(-0.57, 1.14)--(-0.21,1.5);  
 \draw[thick, red, dashed,  rotate=360/7*2, rounded corners=8pt](-1.04,1.08)--(-0.69, 1.35)--(-0.2,1.54);    
 \draw[thick, blue, rotate=360/7*6] (0,1.6) circle [radius=0.3cm]; 
\draw[thick, blue, rotate=0] (0,1.6) circle [radius=0.3cm]; 
 \draw[thick, blue, rotate=360/7, rounded corners=6pt] (0.02,1.8)--(0.1, 2.15)--(0.02,2.5);  
 \draw[thick, blue, dashed, rotate=360/7, rounded corners=6pt]   (-0.02,1.8)--(-0.1, 2.15)--(-0.02,2.5);  
 \draw[thick, dashed, red, rotate=360/7*2, rounded corners=6pt]  (0.02,1.8)--(0.1, 2.15)--(0.02,2.5);  
 \draw[thick, red, rotate=360/7*2, rounded corners=6pt]   (-0.02,1.8)--(-0.1, 2.15)--(-0.02,2.5);  
  \node[scale=0.6, red] at (-2.05,-0.2) {$-$};
  \node[scale=0.6, blue] at (0.8,0.7) {$+$};
  \node[scale=0.6, red] at (0.0,-1.1) {$-$};
  \node[scale=0.6, blue] at (-1.8,1.1) {$+$};
  \node[scale=0.6, blue] at (-0.0,1.1) {$+$};
  \node[scale=0.6, red] at (-0.8,-0.8) {$-$}; 
  \end{scope}

\begin{scope} [xshift=2.9cm, yshift=-5.6cm, rotate=0]
 \draw[very thick, violet] (0,0) circle [radius=2.5cm];
 \draw[very thick, violet] (0,1.6) circle [radius=0.2cm]; 
 \draw[very thick, violet, rotate=360/7] (0,1.6) circle [radius=0.2cm]; 
 \draw[very thick, violet, rotate=360/7*2] (0,1.6) circle [radius=0.2cm];  
 \draw[very thick, violet, rotate=360/7*3] (0,1.6) circle [radius=0.2cm];  
 \draw[very thick, violet, rotate=360/7*4] (0,1.6) circle [radius=0.2cm]; 
 \draw[very thick, violet, rotate=360/7*5] (0,1.6) circle [radius=0.2cm]; 
 \draw[very thick, violet, rotate=360/7*6] (0,1.6) circle [radius=0.2cm]; 
 \draw[thick, red, rotate=360/7*3, rounded corners=8pt](-1.04,1.04)--(-0.57, 1.14)--(-0.21,1.5);   
 \draw[thick, red, dashed,  rotate=360/7*3, rounded corners=8pt] (-1.04,1.08)--(-0.69, 1.35)--(-0.2,1.54);  
 \draw[thick, red, rotate=360/7*2, rounded corners=8pt](-1.04,1.04)--(-0.57, 1.14)--(-0.21,1.5);  
 \draw[thick, red, dashed,  rotate=360/7*2, rounded corners=8pt](-1.04,1.08)--(-0.69, 1.35)--(-0.2,1.54);    
 \draw[thick, blue, rotate=360/7*6] (0,1.6) circle [radius=0.3cm]; 
 \draw[thick, blue, rotate=0, rounded corners=6pt] (0.02,1.8)--(0.1, 2.15)--(0.02,2.5);  
 \draw[thick, blue, dashed, rotate=0, rounded corners=6pt]   (-0.02,1.8)--(-0.1, 2.15)--(-0.02,2.5);  
  \draw[thick, blue, rotate=360/7, rounded corners=6pt] (0.02,1.8)--(0.1, 2.15)--(0.02,2.5);  
 \draw[thick, blue, dashed, rotate=360/7, rounded corners=6pt]   (-0.02,1.8)--(-0.1, 2.15)--(-0.02,2.5);  
 \draw[thick, dashed, red, rotate=360/7*2, rounded corners=6pt]  (0.02,1.8)--(0.1, 2.15)--(0.02,2.5);  
 \draw[thick, red, rotate=360/7*2, rounded corners=6pt]   (-0.02,1.8)--(-0.1, 2.15)--(-0.02,2.5);  
  \node[scale=0.6, red] at (-2.05,-0.2) {$-$};
  \node[scale=0.6, blue] at (0.8,0.7) {$+$};
  \node[scale=0.6, red] at (0.0,-1.1) {$-$};
  \node[scale=0.6, blue] at (-1.8,1.1) {$+$};
  \node[scale=0.6, blue] at (-0.3,2.1) {$+$};
  \node[scale=0.6, red] at (-0.8,-0.8) {$-$}; 
\end{scope}

\begin{scope} [xshift=0cm, yshift=-11.2cm, rotate=0]
 \draw[very thick, violet] (0,0) circle [radius=2.5cm];
 \draw[very thick, violet] (0,1.6) circle [radius=0.2cm]; 
 \draw[very thick, violet, rotate=360/7] (0,1.6) circle [radius=0.2cm]; 
 \draw[very thick, violet, rotate=360/7*2] (0,1.6) circle [radius=0.2cm];  
 \draw[very thick, violet, rotate=360/7*3] (0,1.6) circle [radius=0.2cm];  
 \draw[very thick, violet, rotate=360/7*4] (0,1.6) circle [radius=0.2cm]; 
 \draw[very thick, violet, rotate=360/7*5] (0,1.6) circle [radius=0.2cm]; 
 \draw[very thick, violet, rotate=360/7*6] (0,1.6) circle [radius=0.2cm]; 
 \draw[thick, blue, rotate=0, rounded corners=8pt](-1.04,1.04)--(-0.57, 1.14)--(-0.21,1.5);   
 \draw[thick, blue, dashed,  rotate=0, rounded corners=8pt] (-1.04,1.08)--(-0.69, 1.35)--(-0.2,1.54);  
 \draw[thick, red, rotate=360/7*2, rounded corners=8pt](-1.04,1.04)--(-0.57, 1.14)--(-0.21,1.5);  
 \draw[thick, red, dashed,  rotate=360/7*2, rounded corners=8pt](-1.04,1.08)--(-0.69, 1.35)--(-0.2,1.54);    
\draw[thick, red, rotate=360/7*4] (0,1.6) circle [radius=0.3cm]; 
 \draw[thick, blue, rotate=360/7, rounded corners=6pt] (0.02,1.8)--(0.1, 2.15)--(0.02,2.5);  
 \draw[thick, blue, dashed, rotate=360/7, rounded corners=6pt]   (-0.02,1.8)--(-0.1, 2.15)--(-0.02,2.5);  
 \draw[thick, dashed, red, rotate=360/7*2, rounded corners=6pt]  (0.02,1.8)--(0.1, 2.15)--(0.02,2.5);  
 \draw[thick, red, rotate=360/7*2, rounded corners=6pt]   (-0.02,1.8)--(-0.1, 2.15)--(-0.02,2.5);  
 \draw[thick, dashed, blue, rotate=360/7*6, rounded corners=6pt]  (0.02,1.8)--(0.1, 2.15)--(0.02,2.5);  
 \draw[thick, blue, rotate=360/7*6, rounded corners=6pt]   (-0.02,1.8)--(-0.1, 2.15)--(-0.02,2.5);  
  \node[scale=0.6, red] at (-2.05,-0.2) {$-$};
  \node[scale=0.6, blue] at (1.8,1.0) {$+$};
  \node[scale=0.6, red] at (0.4,-1.1) {$-$};
  \node[scale=0.6, blue] at (-1.8,1.1) {$+$};
  \node[scale=0.6, blue] at (-0.5,1.0) {$+$};
  \node[scale=0.6, red] at (-0.8,-0.8) {$-$};  
\end{scope}

\begin{scope} [xshift=5.8cm, yshift=-11.2cm, rotate=-360/7]
 \draw[very thick, violet] (0,0) circle [radius=2.5cm];
 \draw[very thick, violet] (0,1.6) circle [radius=0.2cm]; 
 \draw[very thick, violet, rotate=360/7] (0,1.6) circle [radius=0.2cm]; 
 \draw[very thick, violet, rotate=360/7*2] (0,1.6) circle [radius=0.2cm];  
 \draw[very thick, violet, rotate=360/7*3] (0,1.6) circle [radius=0.2cm];  
 \draw[very thick, violet, rotate=360/7*4] (0,1.6) circle [radius=0.2cm]; 
 \draw[very thick, violet, rotate=360/7*5] (0,1.6) circle [radius=0.2cm]; 
 \draw[very thick, violet, rotate=360/7*6] (0,1.6) circle [radius=0.2cm]; 
 \draw[thick, blue, rotate=0, rounded corners=8pt](-1.04,1.04)--(-0.57, 1.14)--(-0.21,1.5);   
 \draw[thick, blue, dashed,  rotate=0, rounded corners=8pt] (-1.04,1.08)--(-0.69, 1.35)--(-0.2,1.54);  
 \draw[thick, red, rotate=360/7*2, rounded corners=8pt](-1.04,1.04)--(-0.57, 1.14)--(-0.21,1.5);  
 \draw[thick, red, dashed,  rotate=360/7*2, rounded corners=8pt](-1.04,1.08)--(-0.69, 1.35)--(-0.2,1.54);    
\draw[thick, red, rotate=360/7*4] (0,1.6) circle [radius=0.3cm]; 
 \draw[thick, blue, rotate=360/7, rounded corners=6pt] (0.02,1.8)--(0.1, 2.15)--(0.02,2.5);  
 \draw[thick, blue, dashed, rotate=360/7, rounded corners=6pt]   (-0.02,1.8)--(-0.1, 2.15)--(-0.02,2.5);  
 \draw[thick, dashed, red, rotate=360/7*2, rounded corners=6pt]  (0.02,1.8)--(0.1, 2.15)--(0.02,2.5);  
 \draw[thick, red, rotate=360/7*2, rounded corners=6pt]   (-0.02,1.8)--(-0.1, 2.15)--(-0.02,2.5);  
 \draw[thick, dashed, blue, rotate=360/7*6, rounded corners=6pt]  (0.02,1.8)--(0.1, 2.15)--(0.02,2.5);  
 \draw[thick, blue, rotate=360/7*6, rounded corners=6pt]   (-0.02,1.8)--(-0.1, 2.15)--(-0.02,2.5);  
  \node[scale=0.6, red] at (-2.05,-0.2) {$-$};
  \node[scale=0.6, blue] at (1.8,1.0) {$+$};
  \node[scale=0.6, red] at (0.4,-1.1) {$-$};
  \node[scale=0.6, blue] at (-1.8,1.1) {$+$};
  \node[scale=0.6, blue] at (-0.5,1.0) {$+$};
  \node[scale=0.6, red] at (-0.8,-0.8) {$-$};  
\end{scope}

\begin{scope} [xshift=11.6cm, yshift=-11.2cm, rotate=-360/7]
  \draw[very thick, violet] (0,0) circle [radius=2.5cm];
 \draw[very thick, violet] (0,1.6) circle [radius=0.2cm]; 
 \draw[very thick, violet, rotate=360/7] (0,1.6) circle [radius=0.2cm]; 
 \draw[very thick, violet, rotate=360/7*2] (0,1.6) circle [radius=0.2cm];  
 \draw[very thick, violet, rotate=360/7*3] (0,1.6) circle [radius=0.2cm];  
 \draw[very thick, violet, rotate=360/7*4] (0,1.6) circle [radius=0.2cm]; 
 \draw[very thick, violet, rotate=360/7*5] (0,1.6) circle [radius=0.2cm]; 
 \draw[very thick, violet, rotate=360/7*6] (0,1.6) circle [radius=0.2cm]; 
 \draw[thick, blue, rotate=0, rounded corners=8pt](-1.04,1.04)--(-0.57, 1.14)--(-0.21,1.5);   
 \draw[thick, blue, dashed,  rotate=0, rounded corners=8pt] (-1.04,1.08)--(-0.69, 1.35)--(-0.2,1.54);  
 \draw[thick, red, rotate=360/7*2, rounded corners=8pt](-1.04,1.04)--(-0.57, 1.14)--(-0.21,1.5);  
 \draw[thick, red, dashed,  rotate=360/7*2, rounded corners=8pt](-1.04,1.08)--(-0.69, 1.35)--(-0.2,1.54);   
  \draw[thick, red, rotate=360/7*3, rounded corners=8pt](-1.04,1.04)--(-0.57, 1.14)--(-0.21,1.5);  
 \draw[thick, red, dashed,  rotate=360/7*3, rounded corners=8pt](-1.04,1.08)--(-0.69, 1.35)--(-0.2,1.54);   
 
 \draw[thick, blue, rotate=360/7, rounded corners=6pt] (0.02,1.8)--(0.1, 2.15)--(0.02,2.5);  
 \draw[thick, blue, dashed, rotate=360/7, rounded corners=6pt]   (-0.02,1.8)--(-0.1, 2.15)--(-0.02,2.5);  
 \draw[thick, dashed, red, rotate=360/7*2, rounded corners=6pt]  (0.02,1.8)--(0.1, 2.15)--(0.02,2.5);  
 \draw[thick, red, rotate=360/7*2, rounded corners=6pt]   (-0.02,1.8)--(-0.1, 2.15)--(-0.02,2.5);  
 \draw[thick, dashed, blue, rotate=360/7*6, rounded corners=6pt]  (0.02,1.8)--(0.1, 2.15)--(0.02,2.5);  
 \draw[thick, blue, rotate=360/7*6, rounded corners=6pt]   (-0.02,1.8)--(-0.1, 2.15)--(-0.02,2.5);  
  \node[scale=0.6, red] at (-2.05,-0.2) {$-$};
  \node[scale=0.6, blue] at (1.8,1.0) {$+$};
  \node[scale=0.6, red] at (0.4,-1.1) {$-$};
  \node[scale=0.6, blue] at (-1.8,1.1) {$+$};
  \node[scale=0.6, blue] at (-0.5,1.0) {$+$};
  \node[scale=0.6, red] at (-0.8,-0.8) {$-$};  
  \end{scope}
  
	 \node[scale=0.8] at (-1.5,2.5) {$F_1$};
 	\node[scale=0.8] at (4.3,2.5) {$F_2$};
 	\node[scale=0.8] at (10.1,2.5) {$F_3$};
 	\node[scale=0.8] at (7.2,-3.1) {$F_4$};
 	\node[scale=0.8] at (1.4,-3.1) {$F_5$};
 	\node[scale=0.8] at (-1.5,-8.7) {$F_6$};
 	\node[scale=0.8] at (4.3,-8.7) {$F_7$};
 	\node[scale=0.8] at (10.1,-8.7) {$F_8$};
 	\draw[ ->, rounded corners=5pt] (2.4,1.5)--(2.9,1.6)-- (3.4,1.5);
 	\node[scale=0.8] at (2.9,2) {$R$};
 	\draw[ ->, rounded corners=5pt] (2.4+5.6,1.5)--(2.9+5.6,1.6)-- (3.4+5.6,1.5);
 	\node[scale=0.8] at (2.9+5.6,2) {$F_2F_1$};
 	\draw[ ->, rounded corners=5pt] (11.1,-2.7)--(11, -3.3)--(10.7,-3.7);
 	\node[scale=0.8] at (11.6,-3.2) {$R^{-1}$};
 	\draw[ ->, rounded corners=5pt] (3.3+3,1.5-5.8)--(2.8+3,1.6-5.8)-- (2.3+3,1.5-5.8);
 	\node[scale=0.8] at (5.8,2-5.8) {$F_4F_3$};
 	\draw[ ->, rounded corners=5pt] (2.4,-9.8)--(2.9,-9.7)-- (3.4,-9.8);
 	\node[scale=0.8] at (2.83,-9.3) {$R$};
 	\draw[ ->, rounded corners=5pt] (8.2,-9.8)--(8.7,-9.7)-- (9.2,-9.8);
 	\node[scale=0.8] at (8.7,-9.4) {$F_7F_6$};
\end{tikzpicture}
 	\caption{Proof of Lemma~\ref{lemma:2} for $g=7$.}
\end{figure}

It follows from Theorem~\ref{thm:thmKorkmaz} that
$H=\mod(\Sigma_g)$, completing the proof of the lemma.
 \end{proof}
 
\bigskip

\section{Main Result}

\begin{lemma} \label{lemma:order}
If $R$ is an element of order k in a group $G$ and if $x$ and $y$ are 
elements in $G$ satisfying $RxR^{-1}=y$, then the order of $Rxy^{-1}$ is also k.
In particular, if $\rho$ is an involution element in a group $G$ and if $x$ and $y$ are 
elements in $G$ satisfying $\rho x \rho =y$, then the element $\rho xy^{-1}$ is also an involution.
 \end{lemma}
\begin{proof}
$(Rxy^{-1})^{k}=(yRy^{-1})^{k}=yR^ky^{-1}=1$.\\
On the other hand, if $(Rxy^{-1})^{l}=1$ then $(Rxy^{-1})^{l}=(yRy^{-1})^{l}=yR^ly^{-1}=1$ i.e. $R^l=1$ and hence $k \mid l$.
 \end{proof}
 
Now, we can prove Theorem~\ref{thm:1}.

\begin{proof}
Let $ \rho_3 = R^2 \rho_1R^{-2}$ and $ \rho_4 = R^2 \rho_2R^{-2}$.

Then, by Lemma~\ref{lemma:order} we have 
if $g \geq 7$, then $\rho_3A_1C_1B_3B_5^{-1}C_6^{-1}A_7^{-1}$ is an involution since $\rho_3A_1C_1B_3\rho_3=A_7C_6B_5$ and 
if $g \geq 6$, then $\rho_4A_1C_1B_3B_4^{-1}C_5^{-1}A_6^{-1}$ is an involution since $\rho_4A_1C_1B_3\rho_4=A_6C_5B_4$.

The mapping class group $\mod(\Sigma_6)$ of a closed connected orientable surface of genus 6 is generated by three involutions 
$\rho_1$, $\rho_2$ and \\
$\rho_4A_1C_1B_3B_4^{-1}C_5^{-1}A_6^{-1}$ by Lemma~\ref{lemma:1}.

If $g \geq 7$, then the mapping class group $\mod(\Sigma_g)$ is generated by 
three involutions $\rho_1$, $\rho_2$ and $\rho_3A_1C_1B_3B_5^{-1}C_6^{-1}A_7^{-1}$ by Lemma~\ref{lemma:2}. 
\end{proof}

\end{document}